\documentclass[11pt]{amsart}
\copyrightinfo{2008}{Luke Oeding}

\usepackage{amsmath,amsthm,amsfonts,amscd,amssymb,mathrsfs,cite}
\usepackage[left=1in,top=1in,right=1in]{geometry}

\newtheorem{theorem}{Theorem}[section]

\newtheorem{lemma}[theorem]{Lemma}
\newtheorem{prop}[theorem]{Proposition}

\theoremstyle{definition}

\theoremstyle{remark}
\newtheorem{remark}[theorem]{Remark}
\newcommand{\RR}{\mathbb R}
\newcommand{\CC}{\mathbb C}
\newcommand{\PP}{\mathbb P}

\newcommand{\tdt}{\times\dots\times}
\newcommand{\otdot}{\otimes\dots\otimes}

\DeclareMathOperator{\Sym}{Sym}

\numberwithin{equation}{section}

\begin{document}
\title[Tensors]{Report on Geometry and representation theory of tensors for computer science, statistics and other areas,  July 21 to July 25, 2008 at the American Institute of Mathematics, Palo Alto, California,
organized by Joseph Landsberg, Lek-Heng Lim, Jason Morton, and Jerzy Weyman }

\author{Luke Oeding}
\address{Dept. of Mathematics \\
Mailstop 3368\\
Texas A \&M University\\
College Station, TX  77843-3368}
\curraddr{}
\email{oeding@math.tamu.edu}
\thanks{}


\date{\today}




\maketitle
\section*{Introduction}

This workshop was sponsored by AIM and the NSF and it brought in participants from the US, Canada and the European Union to Palo Alto, CA to work to translate questions from quantum computing, complexity theory, statistical learning theory, signal processing, and data analysis to problems in geometry and representation theory. 

In all these areas, varieties in spaces of tensors that are invariant under linear changes of coordinates appear as central objects of study. Despite their different origins, there are striking similarities among the relevant varieties.  We studied questions such as finding defining equations, hidden symmetries, and singularities.

The AIM workshop was preceded by a graduate student workshop at MSRI on ``Geometry and Representation Theory of Tensors for Computer Science, Statistics, and other areas,''  which was organized by Landsberg, Lim, and Morton.  One goal of this workshop was to introduce the basic tools from geometry and representation theory which are used in studying problems that come up in these other areas. About half of the participants attended both workshops.

At AIM, the mornings were devoted to introductory talks given by L. Valiant, P. Comon, G. Gour, J.-Y. Cai, L. De Lathauwer, and J. Morton.
In the afternoons the work focused on the following 4 areas, however many participants worked with more than one group.

\begin{enumerate}
\item Multi-linear techniques in data analysis and signal processing
\item Geometric approaches to P?=NP
\item Matchgates and holographic algorithms
\item Entanglement in quantum information theory
\end{enumerate}
The participants split into these 4 groups, each of which contained a mix of practitioners and geometers.  The groups then went to work on translating problems into a language which could be understood by both parties.  After the translation work reached a stable point, the two goals were to determine what is already known in each area and to translate the open problems posed by the practitioners into well-posed problems in geometry.  Finally, the participants made first steps toward answers - they determined simpler questions to start working on before the larger problems can be addressed; they made initial calculations to determine what should be true; and they made plans for collaboration focused on these problems.

\section*{Acknowledgments}
The author would like to thank A. Bernardi, A. Boralevi, P.  B\"urgisser,  P. Comon,   L. DeLathauwer,  D. Gross, L. Gurvits, L.H. Kim, J.M. Landsberg, L. Manivel, G. Ottaviani, and  J. Weyman, for the many useful discussions both in person and via email during the preparation of this report.

\part{Signal Processing}
\section{Report}
The decomposition of tensors into a sum of rank-one tensors has retained the attention of those in signal processing for quite a while.  Collaboration with geometers continues to play an important role in determining what is already known and what tools can techniques may be immediately applicable to this and other problems in signal processing.  

An important aspect of the workshop was developing a dictionary to translate notions in signal processing to notions in geometry.   This dictionary opened the lines of communication between the engineers and the geometers.  We will recap some of the relevant definitions here.

Let $V_{1},\dots,V_{n}$ be vector spaces (real or complex) and let $V_{1}\otdot V_{n}$ denote their tensor product.  A \emph{tensor} $T$ is an element of $V_{1}\otdot V_{n}$.  A \emph{tensor decomposition} is the expression of a tensor as the linear combination of other tensors (presumably of lower rank).  The \emph{rank} of a tensor is the minimum $r$ such that $T$ may be expressed as the sum of $r$ rank-one tensors.  

The rank of a tensor is unchanged by scalar multiplication, so it is natural to work in projective space. 
The following are the geometric objects that we will need.

The space of all rank-one tensors is the Segre variety, defined by the embedding,
\begin{eqnarray*} Seg: \PP V_{1} \tdt \PP V_{n} &\rightarrow& \PP \big(V_{1}\otdot V_{n}\big) \\
 ([v_{1}],\dots,[v_{n}]) &\mapsto & [v_{1}\otdot v_{n}]
.\end{eqnarray*}
The space of all rank-one symmetric tensors is the Veronese variety, defined by the embedding,
\begin{eqnarray*} v_{d}: \PP V  &\rightarrow& \PP \big(S^{d}V\big) \\{}
 [w] &\mapsto & [(w)^{d}]
.\end{eqnarray*}
The variety of partially symmetric rank-one tensors is the Segre-Veronese variety defined by the embedding
\begin{eqnarray*} Seg_{d_{1},\dots,d_{n}}: \PP V_{1} \tdt \PP V_{n} &\rightarrow& \PP \big(  S^{d_{1}}V_{1} \otdot S^{d_{n}}V_{n}  \big) \\
 ([w_{1}],\dots,[w_{n}]) &\mapsto & [(w_{1})^{d_{1}} \otdot (w_{n})^{d_{n}}]
.\end{eqnarray*}
Suppose $X\subset \PP V$ is a variety.  The $r^{th}$ \emph{secant variety} to $X$, denoted $\sigma_{r}(X)$, is the Zariski closure of all embedded secant $\PP^{r-1}$'s to X, \emph{i.e.},
\[
\sigma_{r}(X) = \overline{\bigcup_{x_{1},\dots,x_{r} \in X} \PP( span\{x_{1},\dots,x_{r}\})}\subset \PP V
,\]
where the overline indicates Zariski closure.     The \emph{subspace varieties}, denoted $Sub_{r_{1},\dots,r_{n}}(V_{1}\otdot V_{n}) \subset \PP (V_{1}\otdot V_{n})$  are defined in \cite{Landsberg-Morton} ch. 5. as
\[
Sub_{r_{1},\dots,r_{n}}(V_{1}\otdot V_{n}) = \left\{ [T] \in \PP(V_{1}\otdot V_{n}) \mid \begin{array}{c} \forall i \; \exists A_{i}\subset V_{i},\; \dim(A_{i}) = r_{i}, \\{} [T] \in \PP (A_{1}\otdot A_{n}) \end{array} \right\}
.\]
(Note: these are related to   ``rank varieties'' in \cite{Weyman} ch. 7.)  We also consider the \emph{symmetric subspace variety}, 
\[Sub_{r}(S^{n}(V)) := \{[T]\in \PP( S^{n}V )\mid \exists A\subset V, \; \dim(A)=r,\; [T]\subset \PP (S^{n}A)\} 
\subset \PP (S^{n}V)
.\]  
In this setting, we list various notions of rank.
\begin{enumerate}

\item[*] The \emph{rank} of the point $[T] \in \PP (V_{1} \otdot V_{n})$ is the minimum $r$ such that $[T]$ can be expressed as the sum of $r$ points on the Segre variety.
\item[*] More generally, for an algebraic variety $X\subset \PP V$ the \emph{$X$-rank} of a point $p\in \PP V$, denoted $R_{X}(p)$ , is defined by $R_{X}(p) := \min\{r\mid p\in\langle x_{1},\dots,x_{r}\rangle, x_{i}\in X \}$, where $\langle . \rangle$ denotes linear span.
\item[*] The \emph{$X$-border rank} is defined by  $\underline{R}_{X}(p) := \min\{r\mid p\in \sigma_{r}(X)\}$. In particular, $\sigma_{r}(X)$ consists of all points with $X$-border rank no greater than $r$.
\item[*] The \emph{symmetric rank} of a tensor $[T] \in \PP(S^{n}(V))$ is the $X$-rank when $X = v_{n}(\PP V)$.
\item[*] The \emph{partially symmetric rank} of a tensor $[T] \in \PP(V_{1}^{\otimes d_{1}} \otdot V_{n}^{\otimes d_{n}})$ is the $X$-rank when $X = Seg_{d_{1},\dots,d_{n}}(\PP V_{1}\tdt \PP V_{n})$.
\item[*]  The \emph{multilinear rank} of a tensor $[T] \in \PP (V_{1} \otdot V_{n})$ is the minimum $( r_{1},\dots,r_{n} )$ such that $[T] \in Sub_{r_{1},\dots, r_{n}}(V_{1}\otdot V_{n})$. 
\item[*]  The \emph{symmetric multilinear rank}, (deonted $X_{sym}$ below) of a tensor $[T] \in \PP (S^{n}V)$ is the minimum $r$ such that $[T] \in Sub_{r}(S^{n}V)$. 
\item[*] When no $X$ is specified, the convention is that the \emph{rank} (\emph{border rank}) of a tensor $[T]\in \PP( V_{1}\otdot V_{n})$ is the $X$-rank (respectively $X$-border rank) when $X = Seg(\PP V_{1}\tdt \PP V_{n})$.
\item[*] For $X\subset \PP V$, the \emph{generic $X$-rank} is the $X$-rank of a generic tensor in $\PP V$, \emph{i.e.} it is the smallest $r$ such that $\sigma_{r}(X) = \PP V$.
\item[*] A \emph{typical $X$-rank} is an integer $d$ such that the set $\{p \in \PP (V_{1}\otdot V_{n})\mid R_{X}(p) = d \}$ has non-empty interior, \emph{i.e.} it is an $X$-rank that occurs with nonzero probability. Note: over $\RR$ there can be more than one typical $X$-rank.
\end{enumerate}

As far as applications are concerned, two important questions must be answered: (a) in what cases is a tensor decomposition unique, or at least when does there exist a finite number of such decompositions, and (b) how can one reasonably approximate a generic tensor by another of lower rank.  A list of suggested problems from Comon (labeled P\# below) were distributed before the workshop.  During the workshop, the problems below were raised and discussed within the signal processing group. 
\begin{itemize}
\item
\textbf{The Alexander-Hirschowitz theorem and generic ranks:}  
The Alexander-Hirschowitz theorem describes the dimension of $\sigma_{k}(v_{d}(\PP^{N}))$, which is the expected one, with a list of exceptions. This theorem is important because, in particular, it determines the generic symmetric rank of a tensor (see \cite{AH1, AH2, AH3}). G. Ottaviani gave a nice exposition of the proof of Alexander-Hirschowitz, (see also \cite{OttBram2,CGLM}).  

The extension of the Alexander-Hirschowitz theorem to not necessarily symmetric tensors should be the subject of future work (problem P5).  It should be noted that progress in this direction has already been made, see \cite{AOP}.  Very recently, M. V. Catalisano, A.Geramita, and  A.Gimigliano \cite{CGG08} have found that the secant varieties to $Seg(\PP^{1}\tdt \PP^{1})$ are all of the expected dimension with one exception.

\item 
\textbf{Terracini's lemma:} 
Terracini's lemma  \cite{Terracini} is a tool which allows one to compute the generic rank numerically, and is especially useful when theoretical results do not allow one to obtain it.  Terracini's lemma is also often used as a theoretical tool in the Alexander-Hirschowitz theorem, and the work of Chiantini-Ciliberto for example.

 Terracini's lemma also holds true for tensors with Hermitian symmetries, which is useful in Signal Processing when the data are complex (e.g. in Telecommunications) and when the tensor is estimated from the data.  An example was presented by P. Comon regarding cumulant tensors.  In the complex case, and for a given order, there are several cumulant tensors that can be defined. An example of an order 4 cumulant tensor is
 \[T_{ijkl}=Cum\{x_i,x_j,x_k^*,x_l^*\},\]
 where ($*$) indicates complex conjugation. This one has plain symmetries ($ijkl \rightarrow jilk$)  \emph{and}   Hermitian symmetries ($ijkl \rightarrow klij$ or $ijkl \rightarrow lkji$).
Actually, these are the same symmetries as for the 4th order moment $E\{x_i,x_j,x_k^*,x_l^*\}$.  See also  \cite{LathCastCard}.

\item
\textbf{Uniqueness of tensor decomposition:}
Uniqueness of a given tensor decomposition is related to papers of Mella \cite{Mella}, and Chiantini-Ciliberto, \cite{CC02,CC06}. In the cases where the Alexander-Hirschowitz theorem imply that there are at most a finite number of orbits, the work of Chiantini and Ciliberto gives a sufficient condition for uniqueness. Their condition is called weak defectiveness.  This condition is straightforward, but to use it as a test is not a simple calculation.   Mella gives an explicit list of some cases of uniqueness using different methods than Chiantini and Ciliberto.  Certainly there is more work to do for the various types of tensor decompositions related to the other notions of rank mentioned above.

 \item
\textbf{Real tensors (P10):} Tensor decomposition over $\RR$ is more complicated.  In particular, there can be more than one typical rank, but it is not known if there can be more than two typical ranks. This question has remained open after attempts by G. Ottaviani and P. Comon to find counterexamples.
For the study of ranks in the real field, hyperdeterminants and other invariants should be useful.
\item
\textbf{Non-negative tensors (P11):}
Tensors with real positive coefficients are useful in spectral analysis and specifically to medical imaging.
The approximation of a tensor by another of lower rank is well defined for tensors defined on the cone of non-negative reals $\mathbb{R}_+^{\otimes n}$. A proof was sketched by L.-H. Lim.

The approximation of a tensor by another of lower rank is also well-defined when the vectors in at least one mode are constrained to be orthogonal. A proof was sketched by L. De Lathauwer.

A question that came up at the workshop is whether a specific metric is necessary. A metric is mandatory every time we have to deal with an approximation, which may happen for real positive, complex, or real tensors. However, it was claimed that its exact form is not necessarily essential. The question of decomposing a non-negative tensor into a sum of rank-one terms can be posed without a metric. Whether a metric would help to solve the problem is another question. It may be that a topology is sufficient, as for the complex case.

\item
\textbf{Tensors with partial symmetries (P13):}  

Consider the following types of tensors: 
 \subitem{(i)} Tensors in the secant varieties to the Segre-Veronese, $[T] \in \sigma_{k}\big( Seg_{d_{1},\dots,d_{n}}(V_{1}\tdt V_{n}) \big) \subset \PP \big( S^{d_{1}}V_{1}\otdot S^{d_{n}}V_{n}\big)$.
 \subitem{(ii)} 3-tensors enjoying the circular symmetries, \emph{i.e} after a choice of a basis $\{v_{s}^{i}\}$ for each vector space $V_{s}$ we can write $T = T_{ijk}v_{1}^{i}\otimes v_{2}^{j}\otimes v_{3}^{k} \in V_{1}\otimes V_{2}\otimes V_{3}$ and the circular symmetries can be expressed as $T_{ijk}=T_{jki}=T_{kij}$. 
 
 These tensors are encountered in particular when using the joint characteristic function, estimated from data measurements. These are also related to the BIOME and FOOBI algorithms.

(i) was discussed extensively during the workshop, whereas (ii) was only mentioned, but not focused on. It was pointed out that (i) can be useful for the decomposition of the characteristic function, see \cite{ComoR06}, as well as other types of partial symmetries. The example of circular symmetry quoted above was mentioned by L. DeLathauwer.  See also \cite{LathCast}

\item
\textbf{Kruskal's theorem for tensors:}
Given an explicit decomposition of a tensor $T$, Kruskal's theorem provides a test as to whether that decomposition is unique.  A translation into algebraic geometry terminology given by J.M. Landsberg allows one to shorten the proof. Possible generalizations to symmetric and partially symmetric tensors as well as other extensions were discussed.
\item 
\textbf{Comon's Conjecture (P15)}: Do rank and symmetric rank always coincide? In other words, do we have equality in the relation
$$\sigma_k(\nu_d(\mathbb{P}V))\subseteq \sigma_k(Seg(\PP V \tdt \PP V))\cap \mathbb{P}(S^dV).$$

This has been proved for at least two cases: (i) $\operatorname{rank} (T) \leq \operatorname{f-rank}_{[1,\dots,k],[k+1, \dots ,d]}(T)$ and $k<d/2$ (A proof  of this fact, was derived and presented by D. Gross during the workshop. See Lemma \ref{Gross Lemma} below for a definition of $\operatorname{f-rank}_{[1,\dots,k],[k+1,\dots,d]}(T)$ as well as the proof of Gross.), and (ii) under Kruskal's condition $2\,rank(T)\le dn -d+1$, where $n$ denotes the dimension. See also \cite{CGLM}.

For other cases, the problem remains open, despite attempts by A. Bernardi and L. Oeding to find counterexamples.
\item 
\textbf{Rank versus border rank:}
We ask, for which varieties $X$ is it true that the  $R_{X}(p)=\underline{R}_{X}(p)$ for all $p\in \PP(V)$?  In general, rank $\leq k$ is not a closed condition, so it is expected that rank and border rank rarely coincide, (cf. open problem 3).

Note that for any variety $X$, the tangential variety $\tau(X)$ is contained in the secant variety $\sigma(X)$.  When $X$ is the $n$-factor Segre variety, for instance, $\tau(X)$ contains points of rank $n$, however all of these points have border rank 2.  This example shows that there can be an arbitrary discrepancy between rank and border rank.

\item \textbf{Multilinear-ranks:}
L. DeLathauwer introduced a generalization of block decomposition of matrices to block decomposition of tensors. The group pointed out links between this and subspace varieties.  The geometric objects that encapsulate this idea are the secant varieties to the subspace varieties.  We ask, what is the generic $X$-rank when $X$ is the subspace variety?  This question comes down to computing the dimensions of the secant varieties $\sigma_{s}(X)$.

We also want to consider symmetric multilinear rank. We ask, what is the generic $X$-rank when $X = Sub_{r} (S^{k}V)$? Already when $k=3$ these are important questions with great interest to engineers in signal processing.

\item
\textbf{Strassen's (Generalized) Direct Sum Conjecture:}
Suppose $X_{1}= Seg(\PP A_{1}\tdt \PP A_{n})$ and $X_{2} = Seg(\PP B_{1}\tdt \PP B_{n})$, and let $X = Seg\big(\PP( A_{1}\oplus B_{1}) \tdt \PP( A_{n}\oplus B_{n})\big)$.  Suppose $p_{1}\in  A_{1}\otdot A_{n}$, $p_{2}\in B_{1} \otdot B_{n}$ and let $p= p_{1}+p_{2}$. It is known that $R_{X}(p) \leq R_{X_{1}}(p_{1})+R_{X_{2}}(p_{2})$, the conjecture is that equality holds.  

Note:  The original conjecture was stated for bilinear maps, (i.e. the case $n=3$), and B\"urisser, et al. indicate that this conjecture is probably false  (p360 \cite{Burgisser}).  In fact, Sch\"onhage showed that the conjecture is false for border rank, so this seems to be a very subtle problem.  P. B\"urgisser gave a presentation of Sch\"onhage's result during the workshop.

\item
\textbf{Orbits (P6):} For which vector spaces with group actions do there exist finitely many orbits?  In the cases that there are finitely many orbits, what are those orbits explicitly?  J. Weyman pointed out that this problem has been solved. A paper of V. Kac has addressed the problem, \cite{Kac80}, but there are mistakes, corrected in pg 523 \cite{Kac85}, and in the multiplicity free case, \cite{Leahy}. 

\item
\textbf{Explicit decompositions:}  Find explicit algorithms for the decomposition of tensors when the dimension is larger than two (c.f. (P7)). Several suboptimal algorithms exist, for example \cite{lathauwer:295,AlbeFCC04:laa, LathCast, LathCastCard, lathauwer:642}, however they are limited in that the symmetries are not fully exploited and the rank must be smaller than some upper bound, hence, these algorithms cannot decompose generic symmetric tensors.

A suggestion was made by L. Gurvits to resort to doubly stochastic operators, but has not been pursued in depth.
\item
\textbf{Joint decompositions:}  Considering two tensors (of possibly different orders) simultaneously may have two advantages: (i) There may be a possibility to restore uniqueness when the rank expected from the application exceeds the generic rank of the tensors (ii) There may be better robustness in the presence of noise. The problem (cf. (P18)) was discussed (J.M. Landsberg, P. Comon, J. Weyman, G. Ottaviani), but postponed to further collaborations.
\item
\textbf{Sylvester's theorem:}  This theorem was pointed out to be derived earlier by Gaspard de Prony (1755-1839),\cite{Prony}.  This is related to the signal processing problem of the estimation of frequencies of damped sinusoids.
\end{itemize}

It should be mentioned that more generally, all the questions raised for Segre or Veronese varieties (genericity, uniqueness, etc.) can be posed for subspace varieties.

For further reference for the ideas surrounding signal processing, tensor decompositions and questions of rank, we compiled the following list: \cite{AOP, CGG1, CGG2, CGG2.5, Comon1, Comon2, deSilvaLim, LM04, OttBram1,Strassen83 }.

\subsection{From D. Gross}

Let $V=\CC^n$ and $\Sym^d(V)$ be the totally symmetric space in $V^{\otimes d}$. Let $T\in\Sym^d(V)$. By a \emph{decomposition} of $T$, we mean a set $D=\{a_i\}_i$, where $a_{i}^{(k)} \in V $ for each $i,k$, and 
\begin{equation*}
	a_i = a_i^{(1)}\otimes\dots\otimes a_i^{(d)}
\end{equation*}
 is a rank-one tensor that
\begin{equation*}
	T = \sum_i a_i.
\end{equation*}
A decomposition $D$ is \emph{symmetric}, if all elements in $D$ are symmetric.
Let $I\subset\{1,\dots,d\}$. We write
\begin{equation*}
	\operatorname{f-rank}_{I, I^C}(T)
\end{equation*}
for the rank of $T$ viewed as a degree-2 tensor obtained by grouping the respective indices in $I$ and $I^C$ together (the 'f' stands for ``forgetful'' or ``flattening''). Also, if $D$ is a decomposition, write 
\begin{equation*}
	D^{(I)}=\{\bigotimes_{j\in I} a_i^{(j)}\}_i
\end{equation*}
for the projection of the decomposition onto the factors in $I$.  We
get
\begin{lemma}[D. Gross] \label{Gross Lemma}
	Let $T\in S^d(V)$, $d>2$. If, for any $k$,
	\begin{equation*}
		\operatorname{rank} T \leq \operatorname{f-rank}_{[1,\dots,k],[k+1,\dots,d]}(T)
	\end{equation*}
	then any minimal decomposition of $T$ is symmetric.

	What is more, let $D=\{a_i\}_i$ be a decomposition of a symmetric tensor.
	Assume that for any $I\subset\{1,\dots,d\}$ where $|I|=d-2$ the set
	$D^{(I)}$ is linearly independent. Then $D$ is symmetric.
\end{lemma}

\begin{proof}
	We prove the second assertion first. Let $I=\{1,\dots,d-2\}$. By
	assumption, $D^{(I)}$ is linearly independent, so there exists a
	dual basis $\{\alpha_i\}_i$ fulfilling $\alpha_i(a_j)=\delta_{i,j}$ for
	every $a_j\in D^{(I)}$. Hence, for every $i$,
	\begin{equation*}
		\alpha_i(T) = a_i^{(d-1)}\otimes a_i^{(d)}.
	\end{equation*}
	Since $\alpha_i(T)\in\Sym^2(V)$, it follows that
	$a_i^{(d-1)}=a_i^{(d)}$. Now take $I={1,\dots,d-3,d}$. Arguing as
	before, we arrive at $a_i^{(d-2)}=a_i^{(d-1)}$. By
	induction $a_i^{(k)}=a_i^{(l)}$ for all $k,l$, proving the claim.

	The first statement follows, because
	$\operatorname{f-rank}_{I,I^{c}}(T)\leq\operatorname{rank}(T)$ and equality
	trivially implies that $D^{(I)}$ is linearly independent for any
	minimal decomposition $D$.
\end{proof}

Note that the first assertion is formulated in terms of geometric properties of $T$, whereas the second statement just limits our ability to write down non-symmetric decompositions of symmetric tensors. 


\section{Open Problems}

\begin{enumerate}
\item Consider the case when  $X=Sub_{r,\dots,r}(V\otdot V)$, and $X_{sym}=Sub_{r}(S^{n}(V))$. 
Under what general conditions is it true that $X\text{-rank}(T) = X_{sym}\text{-rank}(T)$ for all tensors $T$ in $S^{n}(V)$?  We can ask the same question for border rank. Notice that when $r=1$, $Sub_{1,\dots,1}(V\otdot V) = Seg(\PP V\tdt \PP V)$  and $Sub_{1}(S^{n}(V))= v_{n}(V)$.
Therefore, Comon's conjecture is a special case of this problem.

It was remarked that for $r=1$ some cases are known.  For instance if $dim(V) =2$ and arbitrary $r$, etc.  see also, \cite{Comon-Mourrain}.   There are some very old results, for instance, Prony's method from 1795, which is equivalent to Sylvester's algorithm, \cite{Prony}.

\item 
Consider  $X=Sub_{r_{1},\dots,r_{n}}(V_{1}\otdot V_{n})$, and $X_{sym}=Sub_{r}(S^{n}(V))$. 
What is the generic $X\text{-rank}$ and $X_{sym}\text{-rank}$ for $X$ and $X_{sym}$? Note that  $r_{1}, \dots, r_{n}$, must have $r_{i}\leq r_{j}r_{k}$ for $i,j,k$ distinct, otherwise it collapses.

Suggestion of Weyman:  One can relate this to a question of the existence of an open orbit. There is a close relation to the Castling Transform. Start with $\CC^{p}\otimes \CC^{q}\otimes \CC ^{r}$ with $p\leq q \leq r$ and $r\leq pq$.  Then transform the question to $\CC^{p}\otimes \CC^{q}\otimes \CC^{pq-r}$.  Then you can restrict to an open set and get a correspondence on orbits. Therefore the property of the existence of an open orbit occurs simultaneously for  $(p,q,r)$ and $(p,q,pq-r)$. 

\item Let $X = v_{k}(\PP V)$ and consider $Y_{r} = \{[T] \in S^{k}V\mid X\text{-rank}(T) \leq r \}$ with $k>2$.  Conjecture 6.10 \cite{CGLM} claims that $Y_{r} \neq \overline{Y_{r}}$ for $1<r<R_{S}$, where $R_{S}$ is the minimum $r$ such that $Y_{r} = S^{k}V$.  In particular, this says that the set of symmetric tensors of rank at most $P$ is closed only for $P=1$ and $P=R_{S}$.  What can we say about this conjecture in this case and when $X$ is a subspace variety or a Segre-Veronese variety?

\item What can we say about the uniqueness or finiteness of the corresponding decompositions $t = t_{1} + t_{2} +\dots + t_{d}$, where $t_{i}\in X$, and $X\text{-rank}(t) =d$.  We can ask this question when $X$ is any of the varieties we have mentioned. Particularly interesting are the cases when $X$ is a subspace variety, or a symmetric subspace variety.  Again, many cases are already known, for instance see \cite{Strassen83}.

\item Is the subspace variety version of Strassen's  direct sum conjecture true?  
Let $t_{1}\in A_{1}\otimes B_{1}\otimes C_{1}$,  $t_{2}\in A_{2}\otimes B_{2}\otimes C_{2}$, so that $t_{1} \oplus t_{2 }\in(A_{1}\oplus A_{2}) \otimes 
(B_{1}\oplus B_{2}) \otimes (C_{1}\oplus C_{2}) = A\otimes B\otimes C.$ Let $X_{i} =Sub_{a^{i},b^{i},c^{i}}(A_{i}\otimes B_{i}\otimes C_{i}) \subset \PP (A_{i}\otimes B_{i}\otimes C_{i})$ and $X= Sub_{a,b,c} (A\otimes B\otimes C)\subset \PP (A\otimes B\otimes C)$ be the respective subspace subvarieties. 
We ask if the following equality holds:
\[X\text{-rank}(t_{1}\oplus t_{2}) = X_{1}\text{-rank}(t_{1}) + X_{2}\text{-rank}(t_{2}). \]
What, if any, relationship must there be between $X$, $X_{1}$ and $X_{2}$? What is the status of the conjecture for other varieties $X,X_{1},X_{2}$?  This question can be posed for an arbitrary number of factors as well.

\item The previous 4 questions are posed over the base field $\CC$.  Over $\RR$, what are the typical $X\text{-rank}$'s and  $X_{sym}\text{-rank}$'s in the semi-algebraic sense?

\item Tensor product of cones over orthants: Let $A_{+} = \RR ^{a}_{+}, B_{+} = \RR ^{b}_{+}, C_{+} = \RR ^{c}_{+}$. Let $X_{+}\subset A_{+}\otimes_{\RR_{+}} B_{+}\otimes_{\RR_{+} } C_{+} =: V_{+}$ be a semi-algebraic variety. 
Study the $X_{+}$-rank.  
A tensor with all non-negative coefficients can be written as a non-negative linear combination of tensors of lower rank. In this case we must have $X_{+} = Seg_{+} = (Sub_{1,1,1})_{+}$, and $X_{+}= Seg(\PP A \times \PP B \times \PP C) \cap V_{+}$.
It was remarked that (of course) this rank can be larger than the  $X$-rank.


\end{enumerate}





\part{Mulmuley-Sohoni Approach to $\overline{VP}$ vs. $VP$ vs. $VNP$}
\section{Report}
The Mulmuley and Sohoni approach to P vs. NP  has as its primary goal to ``reduce hard nonexistence problems in complexity theory to tractable existence problems in geometry and representation theory'', \cite{MulSoh0}.  Before we get to this geometry and representation theory, we will briefly review some of the relevant notions.

During the workshop, L. Valiant described algebraic versions of the complexity classes P and NP (these are commonly referred to as VP and VNP), \cite{Valiant79}. Computing the permanent is known to be an VNP-hard problem, however, by using techniques from linear algebra, the determinant can be computed in polynomial time.  Valiant's approach to P vs. NP suggests attempting to show that if one computes the permanent of an $m\times m$ as the determinant of an $n\times n$ matrix, that $n$ must grow faster than a polynomial in $m$.  (For reference, see \cite{ MulSoh1,Mul07}.)

P.  B\"urgisser, J.-Y. Cai pointed out that Toda defined a skew arithmetic circuit (see \cite{Toda91,TodaOg92,Toda92}).  An arithmetic circuit can be described by a class of finite directed acyclic graphs.  All of the edges (usually called wires) are directed towards a single output node.  Each node has either zero or two incoming wires. Nodes with no incoming wires are called inputs and the input to each wire is either a constant or a variable.  Nodes that have two incoming wires are called gates, and each gate does a single elementary computations on those inputs.  
 
In an arithmetic circuit, there are only two types of gates allowed - multiplication and addition.  So each gate takes precisely 2 polynomials as inputs and outputs either the sum or the product of the input polynomials.  A gate may have many identical outputs, but there is only one gate that has no outgoing wire, and this contains the final output of the circuit.  A skew arithmetic circuit carries the additional requirement that every multiplication gate have at least one of its inputs either a variable or a constant.  Locally, a skew arithmetic circuit for multiplication looks like the following.
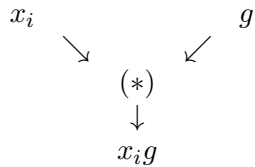
\[
\begin{array}{ccccc}
x_{i} &&&& g \\ & \searrow & & \swarrow  \\ && (*) 
\\ && \downarrow \\ && x_{i}g
\end{array}
\]
The inputs of this circuit are the variable $x_{i}$ and the polynomial $g$ and the output is their product, $x_{i}g$.   The final output of an arithmetic circuit is a polynomial. Toda showed that the determinant of an $n\times n$ matrix can be computed by a skew arithmetic circuit whose size is bounded by a polynomial in $n$.  A sequence of polynomials is in $VP$ if it can be computed by a sequence of skew arithmetic circuits whose sizes and degrees are bounded by a polynomial in the number of inputs.

Toda also defined a notion of weakly skew. Malod and Portier, \cite{MalodPortier} recall that a weakly skew arithmetic circuit is an arithmetic circuit with the following additional requirement: of the two sub-circuits above each multiplication gate, at least one must become disjoint from the rest of the circuit upon deletion of its wire to the multiplication gate.
This may be depicted roughly as follows:
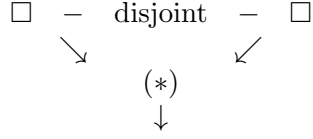
\[
\begin{array}{ccccc}
\square &-&\text{disjoint}&-& \square \\ & \searrow & & \swarrow  \\  && (*) \\ && \downarrow
\end{array}
\]
The boxes represent sub-circuits of the original circuit and $(*)$ is a multiplication gate.
 A sequence of polynomials is in $VP_{WS}$ if it can be computed by a sequence of weakly skew arithmetic circuits whose sizes and degrees are bounded by a polynomial in the number of inputs. These are examples of Valiant's complexity classes. It is not known whether $VP = VP_{WS}$ in general, however, Valiant comments that if you allow $n \log(n)$ - i.e. quasi-polynomial, then all of these classes collapse.  

The sequence $(det_n)$ is complete for the class  $VP_{WS}$ \cite{MalodPortier}.  Since the sequence $(perm_m)$ is VNP complete,  Mulmuley and Sohoni are trying to prove that  $\overline{VP}_{WS} \neq VNP$, where  the class $\overline{VP}_{WS}$ is defined analogously to $\overline{VP}$ as in the paper of B\"urgisser, \cite{Buer:03a}.

Since we don't want to be deceived by bad choices of coordinates, we work with the appropriate orbits of $perm$ and $det$. The Mulmuley-Sohoni approach to $P$ vs $NP$ is to determine if there exists a polynomial $n(m)$ such that 
\[
[l^{n-m}perm_{m}] \in \overline{[GL(W).det_{n}]}
,\]
where $W=Mat_{n\times n}\CC$ and $l$ is a linear form.  We also want to know what we can say about the inclusion
\[  [End(W).det_{n}] \hookrightarrow [\overline{GL(W).det_{n}}].\]
It should be noted that the description of the orbit closure of a given polynomial  is in general very complicated.

The strategy suggested is to understand the homogeneous coordinate rings.  We should try to understand the modules of functions $S\subset \CC [\overline{GL(W).l^{n-m}perm_{m}}]$ which do not occur in $\CC [\overline{GL(W).det_{n}}]$.  Mulmuley-Sohoni call such an $S$ a \emph{representation theoretic obstruction} for $l^{n-m}perm_{m}$ to lie in the variety $\overline{GL(W).det_{n}}$, \cite{MulSoh0}.

\begin{remark}
Let $V$ be a $G$-module for a reductive group $G$, and let $v\in \PP V$.  The we have the following inclusions $G.v\subset \overline{G.v} \subset \PP V$.  The big problem is to understand the restriction $\CC[\overline{G.v}] \hookrightarrow \CC[G.v]$.  
 The only general tool is a theorem of Kostant.
It says that if the closure $\overline{G.v}$ is normal and $\overline{G.v}\setminus G.v$ has codimension 2, then the restriction map is an isomorphism. However in our examples these conditions are almost certainly not satisfied.
Representation theory gives the following decomposition: Let $G_{v}$ denote the stabilizer in $G$ of $v$, then  
\[ 
\CC[G.v] = \CC[G \slash G_{v}]  = \bigoplus_{\lambda\in Irrep(G)} V_{\lambda}^{*}\otimes V_{\lambda}^{G_{v}}
.\] 
This gives information about the coordinate ring of the orbit closure.

\end{remark}

\vspace{.1in}\noindent
\textbf{Step 1:}
The first example is if $v = [det_{n}]$, and $W=M_{n}\CC = A\otimes B$, then 
\[
\CC[GL(W).det_{n}] = \bigoplus K_{\lambda,\delta^{n},\delta^{n}} S_{\lambda}(A\otimes B)
,\] where $K_{\lambda,\delta^{n},\delta^{n}} $ are Kronecker coefficients.

The alternate description is that 
\[
\CC[GL(W). det_{n}] = Sym (\CC^p\otimes \CC^n\otimes \CC^n )^{SL_n\times SL_n}
\]
If $\lambda$ has 2 parts (i.e. $p=2$), then the coefficients $K_{\lambda,\delta^{n},\delta^{n}}$ are known.   This is because the ring of invariants in this case is well known to be isomorphic to the ring of semi-invariants of a Kronecker quiver. see \cite{WeySkow}.

The connection is that $\CC^p\otimes \CC^q\otimes \CC^r$ can be viewed as any of the following:  (i) $p$ copies of $\CC^q\otimes\CC^r$.  (ii) $p$ maps from $\CC^q$ to $\CC^r$.  (iii) the representation space of a quiver with 2 vertices 1,2 and $p$ arrows from vertex 1 to vertex 2, of dimension vector $(q,r)$. This is a generalized Kronecker quiver, for 2 arrows we get just the Kronecker quiver. 
If we look at occurrences of $S_\lambda \CC^p\otimes S_\mu \CC^q\otimes S_\nu \CC^r$ in there, with $\lambda$ having 2 parts, the same representation occurs with the same multiplicity in the coordinate ring, as for $p=2$.

The Kronecker coefficients can be calculated using the Sylvester formula for the covariants of binary forms. Computing the Kronecker coefficients in general is a very difficult problem.  For the computational complexity of these computations, see \cite{BurgIke}.  Immediately after the workshop, Manivel was able to make progress on this problem, \cite{Manivel08}, (cf. problem (4) below).

There are special cases when, $K_{\lambda,\mu,\nu}$, are known to be computable in polynomial time.  They are as follows. (1) one of partitions has  one row. (2) two of the partitions have two rows. (3) two partitions are hooks. (4) one partition is a hook and another has two rows.  (For these results and more, see \cite{Rem89, Rem92, Ros01}).  

We want to find conditions which will determine when $K_{\lambda,\delta^{n},\delta^{n}}$ is zero or not.  A central question is, ``Is this decidability problem in $P$?''  Mumuley conjectures that this decidability problem is in $P$.  However, this  is an open question, \cite{Mul07}.  Another possibly interesting article on this topic is \cite{CHM}.

Computational efforts have determined that $K_{\lambda,\delta^{n},\delta^{n}} $ is almost never zero.  In fact Derksen and Landsberg compared what occurred in the symmetric algebra in each degree.  They did this in order to have a first approximation as to which modules could even be modules of polynomials.  Unfortunately, not only did each module occur, but with high multiplicity.

Ultimately, the success of this line of reasoning will depend on asymptotic questions.  It becomes necessary to describe a certain polytope which is asymptotically nonzero.  There is a hope that this method will be reasonable because the $K_{\lambda,\delta^{n},\delta^{n}} $ are not just generic Kronecker coefficients, they have a special format.

The idea was advanced to use the moment map to establish the vertices and/or faces of the cone
$$C(p,q,r) =\lbrace (\lambda ,\mu ,\nu)\ |\ K_{\lambda ,\mu,\nu}\ne 0\rbrace$$
This is an interesting problem, and L. Manivel described what is known during the workshop.

\vspace{.1in}\noindent
\textbf{Step 2:} The second example  is the permanent $v=perm_{m}$. We need to consider
\[
\CC[GL(W).perm_{m}] = \bigoplus M_{\lambda} S_{\lambda}(A\otimes B),
\]
where $M_{\lambda} = \Sigma_{\mu,\nu, ~} K_{\lambda,\mu,\nu}$, and we require that $\mu$ (resp. $\nu$) is such that $[S_{\mu}A]_{\circ}^{\mathfrak{S}_{n}} \neq 0$  ( resp. $[S_{\nu}B]_{\circ}^{\mathfrak{S}_{n}} \neq 0$).
This is seen to be a more tractable question, and there are a few papers on this.  
For $n=3$, Weyman could show the formula for $[S_{\mu} A]_\circ$ as an $\mathfrak{S}_3$-module.

\vspace{.1in}\noindent
\textbf{Step 3:}
The final question is what happens when we expand to more variables, i.e. when making the transition from $perm_{m}$ to $l^{n-m} perm_{m}$.  A polynomial $f$ is called \emph{stable} under the action of some group $G$ whenever $G.f$ is a closed set in the Zariski topology.  Unfortunately, while $perm_{m}$ is stable under the action of $SL_{m^{2}}$,  the polynomial $l^{n-m} perm_{m}$ is not stable under the action of $SL_{n^{2}}$ (for a short exposition of this concept, see \cite{Agrawal}).

Mulmuley and Sohoni tell us to consider partial stability, or $(R,P)$- stability.  We must enlarge our group action from $GL(W)$, (where $W$ is a $m$-dimensional subspace of $V$), to get a larger stabilizer of the form:

\[
\left(\begin{array}{cc}  Mat_{n\times n}\CC &* \\ 0 &*\end{array} \right)
\]

If we take the natural parabolic, $P\subset GL(W)$, we have a decomposition, $P=KU$, where $U$ is a unipotent radical and $K$ is the Levi factor.  Note that $GL(W)$ is a reductive group.  We choose a reductive subgroup $R\subset K$ which has the same rank (in the sense of the rank of the maximal torus).  Then we choose $v\in \PP V$ such that the following properties hold:

1) $v$ is $R$-stable (The orbit in affine space is closed under the $SL(V)$ action.)

2) $U \subset G_{v}\subset P$

Then Mulmuley and Sohoni have a result which is a partial understanding of $\CC[\overline{G.v}]$ in terms of $\CC[\overline{R.v}]$. Because of the $R$-stability, they can gain some understanding from $\overline{R.v}$.  
This stability allows for a significant reduction: Instead of studying the $G=GL(n^{2})$-module structure of  $\CC[\overline{G.l^{n-m}perm_{m}}]$, we can study the  $R= SL(m^{2})$ -module structure of $ \CC[\overline{R.perm_{m}}]$ (See Theorem 6.2 \cite{MulSoh0} and the remarks following).

The addition of the linear factor $l^{n-m}$ is also related to the Weyman-Kempf geometric technique of calculating syzygies, \cite{Weyman}.
To this end one notices  that the subvariety
$$Y=\lbrace f\in S_d (W)\ | f =l^s g\rbrace$$
where $l$ is a linear form, can be treated by means of the geometric technique.
One just has to consider the incidence variety
$$Z = \lbrace (f,l)\in S_d (W)\times \PP(W)\ \ \ |\ \ \ f=l^s g\rbrace ,$$
which gives a desingularization of the orbit closure. Here we identify the projective space $\PP(W)$ with the set of non-zero linear forms in $W$ (modulo the usual equivalence relation).
There are technical difficulties, as $Y$ is usually not projectively normal.

\section{Open Problems}
\begin{enumerate}

\item In the case $n=3$, compute the difference between $Gl_{n^{2}}.det_{n}$ and $Mat_{n^{2}\times n^{2}}.det_{n}$.  One containment is obvious.  In order to understand why the other containment could fail, consider the following example:  The set $Mat_{2\times 2}. (x^{d}+y^{d})$ does not contain the $GL_{2}$-orbit, $GL_{2}.(x^{d}+y^{d})$ for $d\geq 4$.  This example shows that one must be careful with orbits and closures as the behavior is not always as expected.

In general, compare $\overline{GL_{n^{2}}.det_{n}}$ to $Mat_{n^{2}\times n^{2}}.det_{n}$.  Note that while $Gl_{n^{2}}$ is a group, $Mat_{n^{2}\times n^{2}}$ is an algebraic monoid.

\item (Landsberg) If the two orbits in (1) are not equal, is there a polynomial $p(n)$ such that $Mat_{p(n)^{2}\times p(n)^{2}}.det_{p(n)}$ contains $\overline{l^{p(n)-n}Gl_{n^{2}}.det_{n}}$ for some linear form $l$?
\subitem The affirmative answer to this question would imply that $VP = \overline{VP}$.

\item Let $R = Sym(A \otimes B \otimes C) = \bigoplus_{\lambda,\mu,\nu} \left( S_{\lambda}A \otimes S_{\mu}B \otimes S_{\nu} C\right)^{K_{\lambda \mu \nu}}$, with $K_{\lambda \mu \nu}$ the Kronecker Coefficients.  Let $\Delta(A,B,C) =\{ (\lambda,\mu,\nu)\mid \exists n \text{ such that } S_{n\lambda}A \otimes S_{n\mu} B \otimes S_{n\nu}C \subset R\}$. It is known that  $\Delta(A,B,C)$ is a rational polyhedral cone.  \emph{Determine the facets of this cone.} 

Some cases of this question are already known. For instance in \cite{Franz}, Franz considered the case when $dim(A) = dim(B) = dim(C)$, and completely determined the facets in the $3\times 3\times 3$ case. His method does not necessarily generalize to give all the facets of the cone.  

There are also some partial results by Manivel \cite{Manivel97}, Klyachko \cite{Klyachko}, and Bernstein-Sjamaar \cite{Bernstein00}.

\item  The case when $dim(A) =3$ and $dim(B) = dim(C) = n$.  Compute the ring of invariants,
\[Sym(A \otimes B \otimes C)^{SL(B)\times SL(C)}\]
Note that if $dim(A)=2$, then this is the known Sylvester formula, which is related to the Kronecker coefficients.
This problem would determine $K_{\lambda,\delta^{n},\delta^{n}}$ for $\lambda$ having $\leq 3$ parts, (see \cite{Manivel08}). 

\item Consider $S_{\lambda}A$, where $|\lambda| = \sum \lambda_{i}$ is divisible by $dim(A)$, (i.e. $\lambda$ is in the root lattice).
Let $T$ be the maximal torus of $SL(A)$, i.e. the diagonal matrices with $det=1$. Let $\mathcal{W}$ denote the Weyl group of $SL(A)$. Let $ \left(S_{\lambda} A\right)_{\circ}$ be the zero weight space.    Then $  \left(S_{\lambda} A\right)_{\circ} =  \left(S_{\lambda} A\right)^{T}$ is a representation of $\mathcal{W}$.  For which $\lambda$ does there exist a $\mathcal{W}$ invariant?  In other words, for which $\lambda$ is $\left(S_{\lambda}A\right)_{\circ}^{\mathcal{W}} \neq 0$?

Manivel explained that the answer to this question is known  to be equivalent to the fact that $S_{\lambda} A$ has non zero multiplicity in some $S^d(S^n A)$.

\end{enumerate}

\begin{remark}
Questions 3 and 4 relate to the coordinate ring $\CC[GL_{n^{2}}.det_{n}]$, while questions 3 and 5 relate to the coordinate ring $\CC[GL_{m^{2}}perm_{m}]$. Question 3 leads to an asymptotic understanding, but it is a completely general question.
\end{remark}
\begin{remark}
It is incredible how many important areas come to play.
\end{remark}

\part{Matchgates and Holographic Algorithms}
\section{report}
In this group, the goal was to understand the work of J.-Y. Cai- P. Lu (see \cite{DBLP:conf/coco/CaiL07} for example) and L. Valiant (see \cite{Valiant08}) in light of algebraic geometry and representation theory.

Let $\mathbb{F}$ be a field. If $\Gamma$ is a planar graph with $\gamma$ nodes and $\mathbb{F}$-weighted edges,  the FKT (Fisher-Kasteleyn-Temperley) algorithm counts the number of perfect matchings for $\Gamma$ in polynomial time as follows.   The algorithm assigns an orientation to the graph (by giving a $\pm 1$ to the weight of each edge) and associates to $\Gamma$ a skew symmetric matrix, $k(\Gamma)$.  The Pfaffian of $k(\Gamma)$ counts the number of perfect matchings of $\Gamma$.  The Pfaffian and the determinant are known to be computable in polynomial time via an application of linear algebra.

The goal of matchgates and holographic algorithms is to exploit this polynomial time algorithm to solve problems which were once unknown to be in $P$.  The methods work when they can change a counting problem to one that just requires computing a determinant or Pfaffian.

An example of the use of matchgates is the following.  Consider the problem $\#_{7}$Pl-NAE-ICE which asks the number (modulo 7) of orientations of a planar graph so that there are no sources and no sinks.  The first step is to replace all of the vertices and edges of this graph with signatures (called  generators $G$ and recognizers $R$) and then try to model these signatures with matchgates.  At each node, the signature $(0,1,1,1,1,1,1,0)$ is the truth table values required for the ``not all equal'' condition.  At each edge, the signature $(0,1,1,0)$ is truth table values required by the fact that each edge carries precisely one orientation.  If these signatures can be realized by matchgates, then we will be able to compute $\#_{7}$Pl-NAE-ICE in polynomial time because this procedure changes the question about the old graph to a question about the new graph which can be done in polynomial time and gives the desired answer for the old graph.

For a given problem, what is desired is to know whether there exists a skew symmetric matrix associated to a particular signature.  This can be determined by examining the so-called matchgate identities (MGI).  If the desired signature satisfies MGI, then it is possible to compute that signature as the number of perfect matchings of a graph.  If the desired signature does not satisfy MGI, all is not lost.  Special basis changes are allowed as long as they are done in a way that is consistent with both the generator and the recognizer - more on this in a moment.

If $I \subset [\gamma]$, let $Pf_{2^{I}}$ denote the vector (in $\CC^{2^{|I|}}$) of principal sub-Pfaffians of $k(\Gamma)$ produced by omitting nodes indexed by $J\subset I$ and the adjacent edges.  $\underline{G}$ is a \emph{standard signature} if it occurs as a vector of Pfaffians $Pf_{2^{I}}$ for some graph $\Gamma$ (\emph{i.e.} the MGI are satisfied).  A signature $G$ is \emph{realizable} if there exists a linear map $f:\CC^{c}\rightarrow \CC^{2^{b}}$ so that $\underline{G} =f.G$. Note that this is just a linear map, and not even a group action.  The map $f$ is defined by taking a sufficient tensor power of a vector and implementing it as a matrix. Similarly, a recognizer signature $R$ is realizable if $R=\underline{R}.f$ for some standard signature $\underline{R}$ and the same $f$ as before.  The goal is to find an $f$ so that both $R$ and $G$ are realizable.

This problem has interesting connections to geometry that the participants hope to be able to exploit upon gaining further understanding of these connections.  Consider $Spec((\CC^{2*}))^{k} / MGI) = \$_{k}$.  The MGI are well known to geometers, and the varieties that they cut out have been studied already. There is a canonical identification of the MGI variety, $\$_{k}$, with the spinor varieties, $\mathcal{S}^{k-1}$, given by the inclusion
\[
\mathcal{S}_{+}^{k-1} \setminus \{0\} \sqcup \mathcal{S}_{-}^{k-1} \setminus \{0\} \hookrightarrow \$_{k}.
\]
We also have maps $f:\CC^{c}\rightarrow \CC^{2^{b}}$ locally defined on the wires of the graph.  So, given an explicit piece of data, $\CC^{2^{k}}\otimes \CC^{2^{l}}$, (assume $k<l$) and pairing information $:[k]\hookrightarrow [l]$, and a map $f:\CC^{c}\rightarrow \CC^{2^{k}}$ as above, if $c=2$ then the basis collapse theorem of Cai-Lu tells us that we can reduce the $k$ to $2$.  It is unknown what happens if $c\neq 2$.


\section{Open Problems}
Exercise:  Understand the space of signatures in the product of spinor varieties $\mathcal{S}_{k}\times \mathcal{S}_{l}$.  What is the compatible action?
\begin{enumerate}
\item  Using group actions, reprove the Cai-Lu Basis Collapse Theorem.

The vector space $\CC^{2^{k}}\otimes \CC^{2^{l}}$ has (at least) two useful descriptions.
We can consider the tensor product of spin representations, $\Delta_{k}\otimes \Delta_{l} \simeq \CC^{2^{k}}\otimes \CC^{2^{l}}$,  where $\Delta_{p} = \bigwedge^{even}\CC^{p+1}$: this reflects the action of the Spin groups on the space.
We can also consider the description as tensor powers of $\CC^{2}$, $\left(\CC^{2}\right)^{\otimes k} \otimes \left(\CC^{2}\right)^{\otimes l} \simeq\CC^{2^{k}}\otimes \CC^{2^{l}} $: this reflects the action of the $2\times 2$ matrices.

One may ask: Why not consider a larger space of actions, but the Cai-Lu result reduces to studying the action (actually just a mapping) of $Mat_{2}(\CC)$. 
(Gurvits comments that the spirit of this problem is like quantum computing and unitary operations.)

The goal of this problem is that by rephrasing the proof in terms of geometry and representation theory, this will make the proof of the basis collapse theorem more transparent to geometers, and perhaps it will lead to a greater understanding.  It is hoped that this greater understanding will lead to a proof of the next problem.

\item  Study matchgates with questions about colors instead of just yes/no questions.  Use the results of the previous problem to try to find a natural class of transformation for this setup. The idea is that one may be able to use a different group  ($GL_{4}$ for example) in place of $GL_{2}$.
Does the Cai-Lu Theorem generalize?  The expectation is that it will not, and in that case, this would be considered good news.
This is related to a paper of Valiant which considered 3 colors and saw some success.

More specifically, if $\underbar{G}$ is a standard generator and $\underbar{R}$ is a standard recognizer,  Cai-Lu considered the mapping by $T$, a $2^{k}\times 2$ matrix which sends the standard signatures $\underbar{G}, \underbar{R}$ to desirable non-standard signatures $G,R$.  What if we consider the case that $T$ is a $2^{k}\times 3$ matrix, or a $2^{k} \times 4$, etc, matrix?  This would have interesting applications to theoretical computer science.

\item  Consider $(\CC^{2})^{\otimes k} \otimes ((\CC^{2})^{*})^{\otimes l} \simeq \Delta_{k}\otimes \Delta_{l}$.  We want to understand the diagonal $GL_{2}$-action on points of $\mathcal{S}_{k}$ as an orbit structure.  Many cases are already known.  For instance, Cai pointed out that there is a 2 dimensional space of matrices which leave $\mathcal{S}_{k}$ invariant.  What is desired is to understand and classify the points which have a 1 dimensional space of matrices which move the point off of the spinor variety.  The reason is that if a point has this feature, there will then be 1 dimension's worth of freedom for the each factor in $Seg(\mathcal{S}_{k}\times \mathcal{S}_{l})$.  This in turn will tell which generators have a 1-parameter family of compatible recognizers.
\end{enumerate}

\part{Entanglement and Quantum Information Theory}
\section{Report}
The contributions of G. Gour and D. Gross mention the long-standing additivity conjecture of quantum information theory \footnote{Soon after the workshop, the additivity conjecture for minimum output entropy was proven false by M.B. Hastings, \cite{Hastings}.}. D. Gross posed a question which is closely related to this open problem -- but
has the advantage of being independent of the measure used to quantify entanglement. In particular, there is no need to talk about non-algebraic quantities such as the ``entropy of entanglement''.

The basic setting of the additivity conjecture is a linear space with a two-fold (horizontal and vertical) tensor product structure 
\begin{eqnarray*}
& \left(A_1\otimes B_1\right) \\{}
 \otimes & \left(A_2\otimes B_2\right) .\end{eqnarray*}

Physically, the spaces $A=A_1\otimes A_2$ and $B=B_1\otimes B_2$ are associated with distinct observers, in a way that need not concern us. Let $S_1$ be any subspace of $A_1\otimes B_1$ and likewise $S_2\subset A_2\otimes B_2$. Set $S=S_1\otimes S_2$.

We are interested in the set of vectors in $S$ which display the ``least amount of entanglement'' with respect to the partition $A$ vs. $B$. The physical quantity ``entropy of entanglement'' measures
the deviation of a vector $s\in S$ from being decomposable -- of the form $s_A\otimes s_B$. (The precise definition will not be relevant for the question asked below). The additivity conjecture
amounts to claiming that the set of least entangled vectors always contains one which is decomposable \emph{with respect to the partition $1$ versus $2$}.\footnote{ The conjecture pertains only to the case
where the entropy of entanglement is used to measure the degree of non-decomposibility of a vector. Analogous statements for several other measures -- most notably the tensor rank -- have recently been
proven to be wrong.  } In other words, it is conjectured that if one aims to identify the global minimum of the entropy of entanglement in $S$, one may restrict attention to vectors of the form $s_1\otimes
s_2$, $s_i\in S_i$. It is likely that many researchers in the field would agree to slightly stronger and less precise statement that ``generically'' the global minimum is attained \emph{only} on decomposable vectors.

D. Gross also asks the following consistency question. The optimization problem we set out to solve (find the least entangled vector in $S$) \emph{does not refer to the vertical tensor product structure}.
However, the conjectured statement does. So is the vertical decomposition intrinsic to the problem? Given just the data $A, B$ and $S\subset A\otimes B$, when can one introduce such an additional
tensor product structure and when is it unique?

If a negative answer could be given -- i.e.\ if the tensor product structure is essentially uniquely determined by $S$ -- we would have learned an interesting fact about tensor product spaces. A positive answer would be an \emph{extremely} interesting starting point for the search of a counterexample to the additivity conjecture.

Deciding additivity is considered the most high-profile open problem of quantum information. Any partial result would be met with great interest. There is no compelling evidence either in favor or against the conjecture.

The group working on entanglement and quantum information theory reports that while they did many calculations, they did not solve any of the open problems.

In addition, they computed for $s\in S_{1}\otimes S_{2}$, the quantity $\min\{H(s)\}$, where $H$ can be the entropy or the non-locality for example.  They claim that it is also possible to compute 
\[ \min_{s_{1}\in S_{1}}H(s_{1}) +\min_{s_{2}\in S_{2}}H(s_{2})
,\]
and in so doing, they find $s_{1}^{opt}\otimes s_{2}^{opt}$.  It is unknown whether this is a local optimum.  The following statement has been solved:  If  $s_{1}^{opt}\otimes s_{2}^{opt}$ is a local optimum,  $s_{1}^{opt}\otimes s_{2}^{opt}$ is a critical point, and this implies that locally to first order that additivity works.  Unfortunately, this tells nothing about the global problem.


\section{Open Problems}
\begin{itemize}
\item[(1)] To which degree can one violate the Rank Multiplicitivity (simpler version)?
\end{itemize}
\textbf{Note:}
This problem is perhaps not so interesting as it was aimed at the additivity conjecture  for minimum output entropy which was recently proved to be false.  Nonetheless, we report on the activity from the workshop surrounding this problem.

Consider subspaces $S_{1}\subset A_{1}\otimes A_{2}$, $S_{2}\subset B_{1}\otimes B_{2}$, and let $r(S_{i}) =\min _{0\neq s\in S_{i}}\{rank(s)\}$. Let $S= S_{1}\otimes S_{2}$.  We may consider  $S\subset (A_{1}\otimes A_{2})\otimes ( B_{1}\otimes B_{2})$, or under the natural braiding isomorphism, we can consider $S \subset (A_{1} \otimes B_{1})\otimes (A_{2}\otimes B_{2})$.  It is known that there exist examples where $r(S) \neq r(S_{1})r(S_{2})$.
Find examples of large violations.  (Make it precise what it means to be ``large''.)
It should be noted that this is a toy example.  There is not an automatic correspondence to the big problem for entropy or entanglement.

Following are two examples which were worked out during the workshop.
Let $X\subset Mat_{n\times n}(\RR)$ be a linear subspace of real valued matrices, and let $Y =X\otimes X$. Define $MinRank(X)$ as a minimal rank of nonzero matrices in $X$. Clearly $(MinRank(X))^{2} \geq MinRank(Y)$. 

\begin{prop}[L. Gurvits]
The MinRank - multiplicativity conjecture, $(MinRank(X))^{2} =MinRank(Y)$, is false over $\RR$.  Moreover, there can be an arbitrarily large decrement $(MinRank(X))^2 - MinRank(Y))$ over $\RR$.
\end{prop}
\begin{proof}
Let $I_{k}$ denote the $k \times k$ identity matrix. Consider the matrix, $ M= \left(\begin{array}{cc} 0 & 1\\ -1 & 0 \end{array}\right)$.  Let $X = span\{M,I_{2}\}$.  Clearly, the minimum rank of any nonzero element in $X$ is 2. The spectra $\sigma(M \otimes M) = (1,1,-1,-1)$ and $M \otimes M$ is symmetric. Therefore, we find that $rank(M \otimes M + I_{4}) = rank(M \otimes M -I_{4}) =2 < 4$.  So MinRank multiplicativity is false.

For the arbitrary large decrement, modify the example as follows.  Construct the linear space $S_{2n} = X\otimes I_{n} =span\{ M\otimes I_{n}, I_{2}\otimes I_{n}= I_{2n} \}$.
Then the minimal rank of nonzero matrices in $S_{2n}$ is $2n$, but  $rank ( (M \otimes I_{n}) \otimes (M \otimes I_{n}) - I_{4n^2}) = 2n^2$ and this provides an arbitrarily large decrement.
\end{proof}
\begin{remark}
S. Friedland observed that the same example shows that the entropy additivity conjecture is false over reals since $\log(2n) + \log(2n) > \log(2n^{2})$.
\end{remark}

\begin{itemize}

\item[(2)]  Let $dim(S_{i}) = k_{i}$.  Assume each $S_{i}$ is generic in the sense that it is in general position in $A_{i}\otimes B_{i}$. Is it the case that you can have $r(S) = r(S_{1})r(S_{2})$?  This is \emph not \emph just a dimension counting argument because $S_{1}\otimes S_{2}$ may not be generic.  In fact, it was commented that $S_{1}\otimes S_{2}$ is never generic.  (See Gurvits's example for the real case above).

\item[(3)]  Ambiguity of Tensor Decomposition.  This area hasn't been explored much.

 Are there interesting situations which allow us to find a tensor product structure $A=A_1'\otimes A_2', B=B_1'\otimes B_2'$, ($A_i'\neq A_i$, $B_i'\neq B_i$) and subspaces $S_i'\subset
A_i'\otimes B_i'$ such that $S=S_1'\otimes S_2'$? To qualify as ``interesting'', $S$ should contain no unentangled vector (of the form $s_A\otimes s_B$). Also, it would be desirable if the two decompositions $A=A_1\otimes A_2$ and $A=A_1'\otimes A_2'$ shared no decomposable vector. (There should be some more precise conditions given for a solution to be interesting).

The uniqueness of this decomposition is interesting.  If you prove non-uniqueness, then you become a celebrity as it is related to an important problem in entanglement and quantum information theory.

\end{itemize}

\bibliographystyle{amsplain}
\bibliography{aim_proc}

\providecommand{\bysame}{\leavevmode\hbox to3em{\hrulefill}\thinspace}
\providecommand{\MR}{\relax\ifhmode\unskip\space\fi MR }
\providecommand{\MRhref}[2]{%
  \href{http://www.ams.org/mathscinet-getitem?mr=#1}{#2}
}
\providecommand{\href}[2]{#2}
\begin{thebibliography}{10}

\bibitem{AOP}
H.~Abo, G.~Ottaviani, and C.~Peterson, \emph{Induction for secant varieties of
  {S}egre varieties}, to appear in Trans. AMS, preprint math.AG/0607191.

\bibitem{Agrawal}
M.~Agrawal, \emph{Determinant versus permanent. (english summary)},
  International Congress of Mathematicians. Vol. III, 985--997, Eur. Math.
  Soc., Z\"urich, 2006. 68Q17 (15A15 68W30), MR2275715 (2007k:68038).

\bibitem{AlbeFCC04:laa}
L.~Albera, A.~Ferreol, P.~Comon, and P.~Chevalier, \emph{Blind identification
  of overcomplete mixtures of sources ({BIOME})}, Lin. Algebra Appl.
  \textbf{391} (2004), 3--30.

\bibitem{AH1}
J.~Alexander and A.~Hirschowitz, \emph{La m\'ethode dÕhorace \'eclat\'ee:
  application \`a lÕinterpolation en degr\'ee quatre}, Invent. Math.
  \textbf{107} (1992), 585--602.

\bibitem{AH2}
\bysame, \emph{Polynomial interpolation in several variables}, J. Alg. Geom.
  \textbf{4} (1995), no.~2, 201--222.

\bibitem{AH3}
\bysame, \emph{Generic hypersurface singularities}, Proc. Indian Acad. Sci.
  Math. Sci. \textbf{107} (1997), no.~2, 139--154.

\bibitem{Bernstein00}
A.~Berenstein and R.~Sjamaar, \emph{Coadjoint orbits, moment polytopes, and the
  hilbert-mumford criterion}, J. Amer. Math. Soc. \textbf{13} (2000), no.~2,
  433 --466.

\bibitem{Buer:03a}
P.~B{\"u}rgisser, \emph{The complexity of factors of multivariate polynomials},
  Foundations of {C}omputational {M}athematics \textbf{4} (2004), 369--396.

\bibitem{Burgisser}
P.~B{\"u}rgisser, M.~Clausen, and M.~A. Shokrollahi, \emph{Algebraic complexity
  theory}, Spring{\-}er-Ver{\-}lag, Berlin, 1997.

\bibitem{BurgIke}
P.~B\"urgisser and C.~Ikenmeyer, \emph{The complexity of computing {K}ronecker
  coefficients}, Accepted for FPSAC 2008 http://inst-mat.utalca.cl/fpsac2008.

\bibitem{DBLP:conf/coco/CaiL07}
J.-Y Cai and P.~Lu, \emph{Bases collapse in holographic algorithms}, IEEE
  Conference on Computational Complexity, IEEE Computer Society, 2007,
  pp.~292--304.

\bibitem{CGG08}
M.~V. Catalisano, A.~V. Geramita, and A.~Gimigliano, \emph{Secant varieties of
  $\mathbb{P}^{1}\times\dots\times \mathbb{P}^{1}$ ($n$-times) are not
  defective for $n \geq 5$}, preprint: arXiv:0809.1701v1.

\bibitem{CGG1}
\bysame, \emph{On the rank of tensors, via secant varieties and fat points.
  {Z}ero-dimensional schemes and applications - {N}aples, 2000}, QueenÕs Papers
  in Pure and Appl. Math. \textbf{123} (2002), 133--147.

\bibitem{CGG2}
\bysame, \emph{Ranks of tensors, secant varieties of {S}egre varieties and fat
  points}, Linear Algebra Appl. \textbf{355} (2002), 263--285.

\bibitem{CGG2.5}
\bysame, \emph{Erratum: Ranks of tensors, secant varieties of {S}egre varieties
  and fat points}, Linear Algebra Appl. \textbf{367} (2003), 247--348.

\bibitem{CC02}
L.~Chiantini and C.~Ciliberto, \emph{Weakly defective varieties}, Trans. Amer.
  Math. Soc. \textbf{354} (2002), no.~1, 151--178.

\bibitem{CC06}
\bysame, \emph{On the concept of $k-$sectant order of a variety}, J. London
  Math. Soc. \textbf{73} (2006), no.~2, 436-- 454.

\bibitem{CHM}
M.~Christandl, A.~Harrow, and G.~Mitchison, \emph{On nonzero {K}ronecker
  coefficients and their consequences for spectra}, Commun. Math. Phys.
  \textbf{270} (2007), 575--585, arXiv:quant-ph/0511029v1.

\bibitem{Comon1}
P.~Comon, \emph{Canonical tensor decompositions}, Research Report RR-2004-17,
  I3S (17 June 2004).

\bibitem{Comon2}
\bysame, \emph{Tensors decompositions. {S}tate of the art and applications},
  IMA Conference of Mathematics in Signal Processing (December 18-20, 2000).

\bibitem{CGLM}
P.~Comon, G.~Golub, L-H. Lim, and B.~Mourrain, \emph{Symmetric tensors and
  symmetric tensor rank}, {SIAM} Journal on Matrix Analysis Appl. \textbf{30}
  (2008), no.~3, 1254--1279.

\bibitem{Comon-Mourrain}
P.~Comon and B.~Mourrain, \emph{Decomposition of quantics in sums of powers of
  linear forms}, Signal Processing \textbf{53} (1996), no.~2, 93--107, Special
  issue on High-Order Statistics.

\bibitem{ComoR06}
P.~Comon and M.~Rajih, \emph{Blind identification of under-determined mixtures
  based on the characteristic function}, Signal Processing \textbf{86} (2006),
  no.~9, 2271--2281, http://dx.doi.org/10.1016/j.sigpro.2005.10.007.

\bibitem{LathCast}
L.~De~Lathauwer and J.~Castaing, \emph{Blind identification of underdetermined
  mixtures by simultaneous matrix diagonalization}, Signal Processing, IEEE
  Transactions on \textbf{56} (2008), no.~3, 1096--1105.

\bibitem{LathCastCard}
L.~De~Lathauwer, J.~Castaing, and J.-F. Cardoso, \emph{Fourth-order
  cumulant-based blind identification of underdetermined mixtures}, Signal
  Processing, IEEE Transactions on \textbf{55} (2007), no.~6, 2965--2973.

\bibitem{Prony}
G.R.Baron de~Prony, \emph{Essai \'eŽxperimental et analytique: sur les lois de
  la dila-tabilit\'e de fluides \'elastique et sur celles de la force expansive
  de la vapeur del'alkool, \`a diff\'erentes temp\'eratures}, Journal de
  l'\'Ecole Polytechnique Flor\'eal et Plairial \textbf{1} (1795), no.~22,
  24--76.

\bibitem{deSilvaLim}
V.~de~Silva and L.-H. Lim, \emph{Tensor rank and the ill-posedness of the best
  low-rank approximation problem}, {SIAM} Journal on Matrix Analysis and
  Applications, Special Issue on Tensor Decompositions and Applications, to
  appear.

\bibitem{Franz}
M.~Franz, \emph{Moment polytopes of projective {\bf g}-varieties and tensor
  products of symmetric group representations}, Journal of Lie Theory
  \textbf{12} (2002), no.~2, 539--549.

\bibitem{Hastings}
M.B. Hastings, \emph{A counterexample to additivity of minimum output entropy},
   (2008), http://arxiv.org/abs/0809.3972v2.

\bibitem{Landsberg-Morton}
J.~Morton J.M.~Landsberg, \emph{The geometry of tensors: Applications to
  complexity, statistics and engineering}, in preparation.

\bibitem{Kac80}
V.~Kac, \emph{Some remarks on nilpotent orbits}, J. Algebra \textbf{61} (1980),
  no.~1, 189--918.

\bibitem{Kac85}
V.~Kac and J.~Dadok, \emph{Polar representations}, J. Algebra \textbf{92}
  (1985), 504--524.

\bibitem{Klyachko}
A.~Klyachko, \emph{Quantum marginal problems and representations of symmetric
  groups}, arXiv:quant-ph/0409113.

\bibitem{LM04}
J.M. Landsberg and L.~Manivel, \emph{On the ideals of secant varieties of
  {S}egre varieties}, Found. of Comput. Math. (2004), 397--422.

\bibitem{lathauwer:642}
Lieven~De Lathauwer, \emph{A link between the canonical decomposition in
  multilinear algebra and simultaneous matrix diagonalization}, SIAM Journal on
  Matrix Analysis and Applications \textbf{28} (2006), no.~3, 642--666.

\bibitem{lathauwer:295}
Lieven~De Lathauwer, Bart~De Moor, and Joos Vandewalle, \emph{Computation of
  the canonical decomposition by means of a simultaneous generalized {S}chur
  decomposition}, SIAM Journal on Matrix Analysis and Applications \textbf{26}
  (2004), no.~2, 295--327.

\bibitem{Leahy}
A.~Leahy, \emph{A classification of multiplicity free representations}, J. Lie
  Theory \textbf{8} (1998), no.~2, 367--391.

\bibitem{MalodPortier}
Guillaume Malod and Natacha Portier, \emph{Characterizing {V}aliant's algebraic
  complexity classes}, J. Complex. \textbf{24} (2008), no.~1, 16--38.

\bibitem{Manivel97}
L.~Manivel, \emph{Applications de {G}auss et pl\'ethysme}, Annales de
  l'Institut Fourier \textbf{47} (1997), 715--773.

\bibitem{Manivel08}
\bysame, \emph{A note on certain {K}ronecker coefficients},  (2008), preprint:
  arXiv:0809.3710v1.

\bibitem{Mella}
M.~Mella, \emph{Singularities of linear systems and the {W}aring problem},
  Trans. Amer. Math. Soc. \textbf{358} (2006), no.~12, 5523--5538, MR 2238925
  (2007h:14059).

\bibitem{Mul07}
K.D. Mulmuley, \emph{Geometric complexity theory {VI}: {T}he flip via saturated
  and positive integer programming in representation theory and algebraic
  geometry}, Technical Report TR-2007-04 (2007).

\bibitem{MulSoh1}
K.D. Mulmuley and M.~Sohoni, \emph{Geometric complexity theory {I}: An approach
  to the {P} vs. {NP} and related problems}, SIAM J. Comput. \textbf{31}
  (2001), no.~2, 496--526.

\bibitem{MulSoh0}
K.D. Mulmuley and M.~Sohoni, \emph{Geometric complexity theory, {P} vs. {NP}
  and explicit obstructions}, Advances in algebra and geometry (New Delhi),
  Hindustan Book Agency, 2003, pp.~239--261.

\bibitem{OttBram1}
G.~Ottaviani and C.~Brambilla, \emph{On partial polynomial interpolation},
  preprint arXiv:0705.4448.

\bibitem{OttBram2}
\bysame, \emph{On the {A}lexander-{H}irschowitz theorem}, J. Pure Appl. Algebra
  \textbf{212} (2008), 1229--1251.

\bibitem{Rem89}
J.~B. Remmel, \emph{A formula for the {K}ronecker products of {S}chur functions
  of hook shapes}, J. Algebra \textbf{120} (1989), no.~1, 100Ð118.

\bibitem{Rem92}
\bysame, \emph{Formulas for the expansion of the {K}ronecker products ${S}(m,n)
  \otimes {S}(1^{p-r} ,r)$ and ${S}(1^{k} 2^{l} ) \otimes {S}(1^{p-r} ,r)$},
  Discrete Math. \textbf{99} (1992), no.~1-3, 265Ð 287.

\bibitem{Ros01}
M.~H. Rosas, \emph{The {K}ronecker product of {S}chur functions indexed by
  two-row shapes or hook shapes}, J. Algebraic Combin. \textbf{14} (2001),
  no.~2, 153Ð 173.

\bibitem{WeySkow}
A.~Skowro\'nski and J.~Weyman, \emph{The algebras of semi-invariants of
  quivers}, Transform. Groups \textbf{5} (2000), no.~4, 361--402.

\bibitem{Strassen83}
V.~Strassen, \emph{Rank and optimal computation of generic tensors}, Linear
  Algebra Appl. \textbf{52/53} (1983), 645--685, MR 709378 (85b:15039).

\bibitem{Terracini}
A.~Terracini, \emph{Sulle vk per cui la variet`a degli sh h + 1-secanti ha
  dimensione minore dellÕordinario}, Rend. Circ. Mat. Palermo \textbf{Selecta
  vol. I} (1911), 392--396.

\bibitem{Toda91}
S.~Toda, \emph{{PP} is as hard as the polynomial-time hierarchy}, SIAM J.
  Comput. \textbf{20} (1991), no.~5, 865--877.

\bibitem{Toda92}
\bysame, \emph{Classes of arithmetic circuits capturing the complexity of
  computing the determinant}, IEICE TRANSACTIONS on Information and Systems
  \textbf{E75-D} (1992), no.~1, 116--124, Special Section PAPER (Special
  Section on Theoretical Foundations of Computing).

\bibitem{TodaOg92}
S.~Toda and M.~Ogiwara, \emph{Counting classes are at least as hard as the
  polynomial-time hierarchy}, SIAM J. Comput. \textbf{21} (1992), no.~2,
  316--328.

\bibitem{Valiant79}
L.~Valiant, \emph{Completeness classes in algebra}, STOC '79: Proceedings of
  the eleventh annual ACM symposium on Theory of computing (New York, NY, USA),
  ACM, 1979, pp.~249--261.

\bibitem{Valiant08}
\bysame, \emph{Holographic algorithms}, SIAM J. Comput. \textbf{37} (2008),
  no.~5, 1565--1594.

\bibitem{Weyman}
J.~Weyman, \emph{{Cohomology of Vector Bundles and Syzygies}}, Cambridge Univ.
  Press, Cambridge, 2003, [Cambridge Tracts in Mathematics 149].

\end{thebibliography}

\end{document}